\documentclass[reqno, 12pt]{amsart}

\usepackage{cmap}
\usepackage{a4wide}
\usepackage[T2A]{fontenc}
\usepackage[utf8]{inputenc} 
\usepackage[russian,english]{babel}
\usepackage{amsfonts}
\usepackage{amssymb, amsthm, amscd}
\usepackage{amsmath}
\usepackage{amsaddr}
\usepackage{mathtools}
\usepackage{lipsum}
\usepackage{comment}
\usepackage{etoolbox}
\usepackage[pdftex]{graphicx}
\usepackage[unicode]{hyperref}
\usepackage[matrix,arrow,curve]{xy}
\usepackage[usenames,dvipsnames]{xcolor}
\usepackage{colortbl}
\usepackage{textcomp}

\newcommand{\R}{\mathbb{R}}

\renewcommand{\C}{\mathcal{C}}

\newcommand{\M}{\mathcal{M}}
\newcommand{\F}{\mathcal{F}}

\newcommand{\eps}{\varepsilon}
\newcommand{\q}{\;\,}
\newcommand{\grad}{\mathop{\mathrm{grad}}\nolimits}

\newcommand{\ph}{\varphi}
\newcommand{\Crit}{\mathop{\mathrm{Crit}}\nolimits}
\newcommand{\Morse}{\mathop{\mathrm{Morse}}\nolimits}

\newcommand{\codim}{\mathop{\mathrm{codim}}\nolimits}
\newcommand{\im}{\mathop{\mathrm{im}}\nolimits}
\newcommand{\sign}{\mathop{\mathrm{sgn}}\nolimits}

\newcommand{\Hess}{\mathop{\mathrm{Hess}}\nolimits}

\renewcommand{\tilde}{\widetilde}

\newcommand{\res}[2]{\left. #1 \right|_{#2}}

\theoremstyle{plain}
\newtheorem{thm}{Theorem}[section]

\newtheorem{lem}[thm]{Lemma}

\newtheorem{prop}[thm]{Proposition}

\newtheorem{prob}[thm]{Problem}
\newtheorem{conj}[thm]{Conjecture}

\theoremstyle{definition}
\newtheorem{defn}[thm]{Definition}
\newtheorem{cor}[thm]{Corollary}

\theoremstyle{remark}
\newtheorem{rem}[thm]{Remark}

\title{On Morse index retrieval}
\author{Daniil Mamaev}
\date{\today (Last Typeset)}
\address{St. Petersburg Department of V.A. Steklov Mathematical Institute, Russia\\
Chebyshev Laboratory at St. Petersburg State University, Russia}
\email{dan.mamaev@gmail.com}

\begin{document}
\selectlanguage{english}

\begin{abstract}
A smooth function $f$ in a neighbourhood of the unit sphere $S^{n - 1}$ is said to admit index $\lambda$ if it can be extended to a function $F$ in the unit ball $B^n$ such that $F$ has a unique critical point $p$ and the Morse index of $p$ is equal to $\lambda$. It is easy to see that a function $f$ cannot admit two indices of different parity. We prove that for any two indices that differ by two there exists a function $f$ that admits both of them.
\end{abstract}

\maketitle

\section{Introduction}
Consider a smooth function $f$ in a neighbourhood of the unit sphere ${S^{n - 1} = \{x \in \R^n \colon |x| = 1\}}$. Its smooth extensions inside the unit ball ${B^n = \{x \in \R^n \colon |x| \le 1\}}$ generically have only non-degenerate critical points. In \cite{Bar94} S.Barannikov, following a question by V.I. Arnold, gave a lower bound for the number of critical points of such extensions making use of what later became known as Morse-Barannikov complexes. In the present paper we consider a related problem.

\begin{prob} \label{the problem}
Suppose we are given a smooth function $f$ in a neighbourhood $U$ of the unit sphere $S^{n - 1}$ inside the unit ball $B^n_{\mathbf 0}(1)$. Let $F\colon B_{\mathbf 0}^n(1) \to \R$ be a smooth extension of $f$ such that the origin $\mathbf 0$ is the unique critical point of $F$ and $\Hess_{\mathbf 0} F$ is non-degenerate. What information about the Morse index $\mu_F(\mathbf 0)$ can be retrieved from $f$?
\end{prob}

We say that a function $f$ defined in a neighbourhood $U$ of the unit sphere $S^{n - 1}$ inside the unit ball $B^n_{\mathbf 0}(1)$ \emph{admits} index $\lambda$ if there exists a smooth extension $F\colon B_{\mathbf 0}^n(1) \to \R$ of $f$ such that the origin is the unique critical point of $F$ and its Morse index is equal to $\lambda$. There are cases when a function admits only one index. For instance, if $\res{\grad f}{S^{n - 1}}$ always points inside the ball, then $\mathbf 0$ must be the point of global maximum of $F$ hence $\mu_F(\mathbf 0) = n$. In general, the parity of $\mu_F(\mathbf 0)$ can be retrieved from $f$ (see Proposition~\ref{parity of the index}). 

The main result of the paper is that \emph{for any $n \ge 2$ and $0 \le \lambda \le n - 2$ there exists a function $f$ defined in a neighbourhood of $S^{n - 1}$ that admits both index $\lambda$ and $\lambda + 2$} (see~Theorem~\ref{two functions}).

As for the general answer to Problem~\ref{the problem}, the following conjecture seems plausible.

\begin{conj} \label{Conjecture}
For any $n \ge 2$ there exist functions $f_0$ and $f_1$ in a neighbourhood of $S^{n - 1}$ such that $f_0$ admits indices $0, 2, \ldots, 2 \cdot [n/2]$ and $f_1$ admits indices $1, 3, \ldots, 2 \cdot [(n + 1)/2] - 1$.
\end{conj}

The paper is organised as follows. In Section~\ref{Preliminaries} we briefly review the basics of Morse and Cerf's theories and apply it to prove that the parity of $\mu_F(\mathbf 0)$ can always be retrieved from $f$, the goal of this section is mainly to fix the notation. Sections~\ref{Toolbox}~and~\ref{Functions admitting different indices} are devoted to the proof of Theorem~\ref{two functions}. In Section~\ref{Toolbox} we develop a number of tools to perform the metamorphoses of functions in a controllable manner. In Section~\ref{Functions admitting different indices} we first explain the two-dimensional case and then prove the theorem.
\medskip
\\
{\bf Acknowledgement.}
I am deeply indebted to Gaiane Panina for posing the problem and supervising my research. I am also grateful to Serguei Barannikov for useful comments. The first known to me example of a function $f$ in dimension $2$ that admits indices $0$ and $2$ was constructed by Sem\"en Podkorytov (private communication). This paper is based on my Master's thesis defended at St. Petersburg State University. The research is supported by the Russian Science Foundation under Grant 21-11-00040.

\section{Preliminaries: Morse and Cerf's theories}
\label{Preliminaries}
In this section we mainly follow two books by Milnor \cite{Milnor65}, \cite{Milnor69} and a more modern exposition by Nicolaescu \cite{Nic11}. 

Throughout the text smooth means $C^\infty$, a manifold means a manifold with boundary, and a closed manifold means a compact manifold without boundary. Let $M^n$ be a smooth $n$-dimensional manifold with the boundary $\partial M$. For a boundary point $p \in \partial M$ the tangent space $T_p \partial M$ divides the tangent space $T_p M$ into two semi-spaces that consist of the vectors pointing inside or outside the manifold; these (open) semi-spaces are denoted by $T_p^{in} M$ and $T_p^{out} M$ respectively. 

For a smooth fibre bundle $E \to M$ and a subset $X \subseteq M$ we denote by $C^{\infty}(X, E)$ the space of smooth sections of $E$ over $X$. That is, the space of (global) smooth vector fields on $M$ is $C^\infty(M, TM)$, the space of smooth real-valued functions $C^{\infty}(M, M \times \R)$ is abbreviated to $C^{\infty}(M)$. 

Given a vector field $V \in C^\infty(M, TM)$ and a smooth function $f \in C^\infty(M)$ the \emph{derivative of $f$ along $V$} is the function $Vf \in C^\infty(M)$ given by $Vf(p) = d_p f(V(p))$ for any $p \in M$.

\subsection{Morse functions}

Let ${f\colon M \to \R}$  be a smooth function. A point $p \in M$ is called a \emph{critical point} of $f$ if the differential $d_p f\colon T_pM \to \R$ vanishes, otherwise $p$ is called a \emph{regular point} of $f$. The point $p$ is critical if and only if in (any hence all) local coordinates $(x^1, \ldots, x^n)$ around $p$ the partial derivatives $\frac{\partial f}{\partial x^i}$ vanish at $p$. $\Crit f$ denotes the set of all critical points of $f$. 

Let $p$ be a critical point of the function $f$. The \emph{Hessian form} of $f$ at $p$ is a symmetric bilinear form $\Hess_p f \colon T_pM \times T_p M \to \R$ defined on a pair of tangent vectors $u, v \in T_pM$ by
$$
\Hess_p f(u, v) = \big(U(Vf)\big)(p), 
$$
where $U$ and $V$ are arbitrary extensions of $u$ and $v$ to local vector fields around $p$. In local coordinates $(x^1, \ldots, x^n)$ around $p$ one has 
$$
\Hess_p f \left(\frac{\partial}{\partial x^i}, \frac{\partial}{\partial x^j}\right) =  \frac{\partial ^ 2 f}{\partial x^i \partial x^j} (p).
$$
The $n \times n$ matrix $H_p f$ with entries $(H_p f)_{i, j} =  \frac{\partial ^ 2 f}{\partial x^i \partial x^j} (p)$ is called the \emph{Hessian matrix} of $f$ at $p$ in the local coordinates $(x^1, \ldots, x^n)$.

A point $p \in \Crit f$ is called a \emph{degenerate} critical point of the function $f$ if the Hessian form $\Hess_p f$ is degenerate (that is, there exists a vector $v \in T_p M$ such that $\Hess_p f(v, \underline{\phantom{u}}) \colon T_pM \to \R$ vanishes). Otherwise $p$ is called a \emph{non-degenerate} (or \emph{Morse}) critical point of the function $f$. The point $p$ is degenerate if and only if in local coordinates $(x^1, \ldots, x^n)$ around $p$ the Hessian matrix $H_p f$ has a zero eigenvalue. $\Morse f$ denotes the set of all Morse critical points of $f$.

For a point $p \in \Morse f$ its \emph{Morse index} $\mu_f(p)$ is the maximal dimension of a subspace of $T_p M$ on which the Hessian form $\Hess_p f$ is negative definite, that is,
$$
\mu_f (p) = \max \left\{\dim V \colon V \leq T_p M \text{ and } \forall v \in V \setminus \{0\} \q \Hess_p f(v, v) < 0\right\}.
$$
The Morse index $\mu_f(p)$ is equal to the number of negative eigenvalues of the Hessian matrix $H_p f$ in local coordinates $(x^1, \ldots, x^n)$ around $p$.
\begin{defn}
A smooth function $f\colon M \to \R$ on the smooth manifold $M$ with the boundary $\partial M$ is called a \emph{Morse function} if
\begin{enumerate}
\item $\Crit f \subset M \setminus \partial M$, that is, there are no critical points of $f$ in the boundary $\partial M$,
\item $\Morse f = \Crit f$, that is, the critical points of $f$ are non-degenerate, and
\item $\Morse \res{f}{\partial M} = \Crit \res{f}{\partial M}$, that is, the critical points of the restriction of $f$ to the boundary $\partial M$ are non-degenerate.
\end{enumerate}
\end{defn}

\subsection{Flow lines}

\begin{defn} \label{gradient-like vector field}
Let $f$ be a Morse function on a smooth manifold $M$.
\begin{enumerate}
\item A local coordinate system $(x_1, \ldots, x_n)$ in a neighbourhood $U_p$ of a point $p \in \Crit f$ is \emph{adapted to $f$} if 
$$
f = f(p) - \left(x^1\right) ^ 2 - \ldots - \left(x^\lambda\right)^2 + \left(x^{\lambda + 1}\right)^2 + \ldots + \left(x^n\right) ^ 2 \text{ in } U_p
$$
\item A  vector field $V \in C^\infty(M, TM)$ is called a \emph{gradient-like vector field for $f$}  if $Vf (p) > 0$ for all non-critical $p \in M$, and for any critical point $p$ of $f$ there exists a neighbourhood $U_p$ of $p$ and a coordinate system $(x^1, \ldots, x^n)$ in $U_p$ adapted to $f$ such that ${V/2 = -x^1 \frac{\partial}{\partial x^1} - \ldots -x^\lambda \frac{\partial}{\partial x^\lambda} +x^{\lambda + 1} \frac{\partial}{\partial x^{\lambda + 1}} + \ldots +x^n \frac{\partial}{\partial x^n}}$ on $U_p$.
\item A Riemannian metric $g \in C^{\infty}(M, S^2T^*M)$ is \emph{adapted to $f$} if for any critical point $p$ of $f$ there exists a neighbourhood $U_p$ of $p$ and a coordinate system $(x^1, \ldots, x^n)$ in $U_p$ adapted to $f$ such that $g = \left(dx^1\right)^2 + \ldots + \left(dx^n\right)^2$ on $U_p$.
\end{enumerate}
\end{defn}

\begin{lem}[Morse]
Let $f$ be a smooth function on a smooth manifold $M$ and $p$ be a non-degenerate critical point of $f$. Then there exists a neighbourhood $U$ of $p$ and local coordinates $(x_1, \ldots, x_n)$ on $U$ adapted to $f$.
\end{lem}
\begin{rem} Let $f$ be a Morse function on a smooth manifold $M$ with boundary $\partial M$. Riemannian metrics on $M$ adapted to $f$ and gradient-like vector fields for $f$ are closely related. Namely,
\begin{enumerate}
\item Given a Riemannian metric $g$ on $M$ adapted to $f$ (their existence is easily deduced by a partition of unity argument), the vector field $V = \grad_g f$ is called the \emph{gradient-like vector field for $f$ associated to $g$}.
\item Given a gradient-like vector field $V$ for $f$ one can define a Riemannian metric $g$ on $M$ adapted to $f$ such that $\grad_g f = V$ in the following fashion. Fix neighbourhoods $U_p$ from the definition of the gradient-like vector field. Define Riemannian metrics $g_p \in C^{\infty}(U_p, S^2T^*M)$ by $g_p = \left(dx^1\right)^2 + \ldots + \left(dx^n\right)^2$. Take a positive definite bilinear form $\tilde g \in C^{\infty}(M\setminus \Crit f, S^2 (\ker df)^*)$ and define a Riemannian metric $g_{reg} \in C^{\infty}(M\setminus \Crit f, S^2 T^*M)$ by 
$$
g_{reg}(u + aV(x), v + bV(x)) = \tilde g(u, v) + ab \cdot (Vf)(x) \text{ for } u, v \in \ker d_x f \text{ and } a, b \in \R.
$$
The metric $g$ obtained from $g_p$ and $g_{reg}$ using a partition of unity subordinate to the open cover $\{U_p\}_{p \in \Crit f} \cup \{M \setminus \Crit f\}$ is the one we need.
\end{enumerate}

\end{rem}

Now let $M$ be a closed manifold, $f\colon M \to \R$ be a Morse function and $V \in C^\infty(M, TM)$ be a gradient-like vector field for $f$.  Denote by $\Phi_t$ the flow on $M$ determined by $-V$, that is, 
$$
\left.\frac{d}{dt}\right|_{t = t_0} \Phi_t(x)= -V(\Phi_{t_0}(x)) \text{ and } \Phi_0(x) = x \text{ for any } x \in M \text{ and } t_0 \in \R. 
$$
For any point $x \in M$ the limits $\Phi_{\pm \infty}(x) = \lim\limits_{t \to \pm \infty} \Phi_t(x)$ exist and are critical points of $f$. For a point $x \in M$ the curve $\gamma_V^x = \Phi_{\underline{\phantom{t}}} (x) \colon \R \to M$ is called the \emph{parametrised flow line} through $x$ (with respect to $V$) and its image is called the (unparametrised) \emph{flow line} through $x$ (with respect to $V$).

For a point $p \in \Crit f$ we set
$$
W_{p}^{\pm} = W_{p}^{\pm}(V) := \Phi^{-1}_{\pm \infty}(p) = \left\{x \in M \colon \lim_{t \to \pm \infty}  \Phi_t(x) = p\right\}.
$$
The sets $W_p^{+}$ and $W_p^-$ are called the \emph{stable} and \emph{unstable manifolds} of $p$ with respect to $V$ respectively. 
\begin{defn}
Let $f\colon M \to \R$ be a Morse function on a smooth closed manifold $M$ and $V \in C^\infty(M, TM)$ be a gradient-like vector field for $f$. $V$ is called a \emph{Morse-Smale vector field} adapted to $f$ if for any $p, q \in \Crit f$ the unstable manifold $W_{p}^-(V)$ intersects the stable manifold $W_{q}^+(V)$ transversally.
\end{defn}

\begin{thm}[Smale, \cite{Smale61}]
For any Morse function $f$ on a smooth closed manifold $M$ there exists a Morse-Smale vector field on $M$ adapted to $f$. 
\end{thm}

Let $f\colon M \to \R$ be a Morse function on the closed manifold $M^n$ and let $V \in C^{\infty}(M, TM)$ be a Morse-Smale vector field adapted to $f$. Consider two points $p, q \in \Crit f$. The intersection $W_q^p = W_p^- \cap W_q^+$ consists of the flow lines with source $q$ and target $p$ and its dimension is
$$
\dim W_q^p = n - (\codim_M W_p^- + \codim_M W_q^+) = n - (\mu_f(p) + n - \mu_f(q)) = \mu_f(p) - \mu_f(q).
$$
The space $W_q^p$ is endowed with a free action of $\R$ given by the flow $\Phi$. The quotient $\M_q^p = W_q^p/\R$ is obviously in bijection with the flow lines going from $p$ to $q$ and is thus called the \emph{moduli space of flow lines} from $p$ to $q$. 

\begin{prop}~
\begin{enumerate}
\item Let $f(q) < r < f(p)$. Then $W_q^p \cap f^{-1}(r)$ is a smooth submanifold of $M$ of dimension $\mu_f(p) - \mu_f(q) - 1$.
\item The moduli space $\M_q^p$ is a smooth manifold diffeomorphic to any of $W_q^p \cap f^{-1}(r)$ with $f(q) < r < f(p)$.
\end{enumerate}
\end{prop}

Now let $p$ and $q$ be critical points of indices $\lambda$ and $\lambda - 1$ respectively. Then $\M_q^p$ is a $0$-dimensional compact manifold, that is, a finite set with discrete topology. For each flow line $\bar \gamma \in \M_q^p$ we define its \emph{sign} $\sign \bar \gamma$ to be $\pm 1$ depending on the orientation of a frame at a point $x \in \bar \gamma \cap f^{-1}(r)$, consisting of positively oriented frames of $W_p^- \cap f^{-1}(r)$ and $W_q^+ \cap f^{-1}(r)$ at point $x$ together with $V(x)$. 

\subsection{Cerf's theory}
For a more detailed exposition of Cerf's theory see \cite{Cerf70} and \cite{Sharko93}. 

Let $f, g\colon M \to \R$ be two smooth functions on a closed manifold $M$. They are called \emph{equivalent} if there are diffeomorphisms $R \colon M \to M$ and $L \colon \R \to \R$ such that $f = L \circ g \circ R^{-1}$. A Morse function is called \emph{non-resonant} if all of its critical values are distinct. A Morse function is called \emph{simply resonant} if the number of its distinct critical values differs by one from the number of its critical points. A smooth function $f\colon M \to \R$ is called a \emph{birth-death} function if all of its critical points but one are Morse, all the critical values are distinct, and there is a local coordinate system around the only non-Morse point $p$ in which the function is expressed as
$$
f = f(p) - \left(x^1\right) ^ 2 - \ldots - \left(x^\lambda\right)^2 + \left(x^{\lambda + 1}\right)^2 + \ldots + \left(x^{n - 1}\right) ^ 2 + \left(x^n\right) ^ 3.
$$
\begin{prop}~
\begin{enumerate}
\item Let $f$ be a non-resonant Morse function. Then there is a neighbourhood $U$ of $f$ such that each $g \in U$ is a non-resonant Morse function equivalent to $f$.
\item Let $f$ be a simply resonant Morse function. Then there is a neighbourhood ${U = U_> \sqcup U_= \sqcup U_<}$ of $f$ such that $U_=$ is a codimension-one submanifold of $U$ consisting of simply resonant Morse functions equivalent to $f$, while $U_>$ and $U_<$ are open subsets consisting of equivalent non-resonant Morse functions.
\item Let $f$ be a birth-death function. Then there is a neighbourhood ${U = U_0 \sqcup U_1 \sqcup U_2}$ of $f$ such that such that $U_1$ is a codimension-one submanifold of $U$ consisting of birth-death Morse functions equivalent to $f$, while $U_0$ and $U_2$ are open subsets consisting of equivalent non-resonant Morse functions. $\# \Crit g = \# \Crit f - 1$ for $g \in U_0$ and $\# \Crit g = \# \Crit f + 1$ for $g \in U_2$.
\end{enumerate}
\end{prop}

Let us denote by $\F_0$ the set of non-resonant Morse functions, by $\F_1^\alpha$ the set of birth-death functions and by $\F_1^\beta$ the set of simply resonant functions. $\F_0$ is an open dense subset of $C^\infty(M, \R)$ (in $C^2$-topology), $\F_1^\alpha$ and $\F_1^\beta$ are codimension-one (Frechet) submanifolds of $C^\infty(M, \R)$.

\begin{prop}
Let $f_0, f_1 \colon M \to \R$ be two non-resonant Morse functions. Then there exists a path $\gamma\colon [0, 1] \to \C^\infty(M, \R)$ such that
\begin{enumerate}
\item $\gamma(0) = f_0$ and $\gamma(1) = f_1$;
\item $\gamma(t)$ is a non-resonant Morse function, simply resonant Morse function or a birth-death function for all $t \in [0, 1]$;
\item $\gamma(t)$ intersects $\F_1^\alpha$ and $\F_1^\beta$ transversally (and thus in a finite number of points). 
\end{enumerate}
Moreover, a generic path satisfying (1) satisfies (2) and (3).
\end{prop}

\begin{defn}
A Morse function $F\colon M \to \R$ on the smooth manifold $M$ with the boundary $\partial M$ is called \emph{non-resonant} if all the critical values of $F$ and of $\res{F}{\partial M}$ are distinct.
\end{defn}

\subsection{The parity of the Morse index}
Let $f\colon U \to \R$ be a Morse function without critical points in the neighbourhood $U$ of $S^{n - 1}$ inside $B_{\mathbf 0}^n(1)$. For each $p \in S^{n - 1}$ the vector $d_p f \ne 0$ and if $p$ is a critical point of $\res{f}{S^{n - 1}}$, then $d_p f = T_p\partial S^{n - 1}$ since otherwise $p$ would be a critical point of $f$. We introduce the following notation:
\begin{align*}
\Crit_\lambda(f) &= \{p \in \Crit f \colon \mu_f(p) = \lambda\};\\
\Crit^{out}(f) &= \{p \in \Crit f \colon \res{d_p f}{T_p^{out}M} > 0\}, \q 
\Crit^{in}(f) = \{p \in \Crit f \colon \res{d_p f}{T_p^{in}M} > 0\};\\
\Crit_\lambda^{out}(f) &= \Crit_\lambda(f) \cap \Crit^{out}(f), \q \Crit_\lambda^{in}(f) = \Crit_\lambda(f) \cap \Crit^{in}(f).
\end{align*}
\begin{prop} \label{parity of the index}
Let $F$ be a smooth extension of $f$ to $B_{\mathbf 0}^n(1)$ such that ${\Crit F = \Morse F = \mathbf 0}$. Then the parity of $\mu_F(\mathbf 0)$ can be retrieved from $f$
\end{prop}
\begin{proof}
Let $V \in C^{\infty}(S^{n - 1}, TB_{\mathbf 0}^n(1))$ be a gradient-like vector field for $f$ (along $S^{n - 1}$). Consider a local coordinate system $(x^1, \ldots, x^n)$ around $\mathbf 0$ from Definition~\ref{gradient-like vector field} and a small sphere $S_0$ given by $\left(x^1\right) ^ 2 + \ldots + \left(x^n\right) ^ 2 = \eps$. Connect $S^{n - 1}$ with $S_0$ by an isotopy $I\colon S^{n - 1} \times [0, 1] \to B_{\mathbf 0}^n(1)$. Then $v_t(x) = V(I(x, t)) / |V(I(x, t))|$ defines a homotopy between $v_1, v_0\colon S^{n - 1} \to S^{n - 1}$ hence $\deg v_1 = \deg v_0$. It can be easily computed that $\deg v_0 = (-1) ^ {\mu_F(\mathbf 0) + n}$, thus $V$ determines the parity of $\mu_F(\mathbf 0)$.
\end{proof}
\begin{rem}
In fact, applying Cerf's theory (and its restatement by Barannikov \cite{Bar94}) one can show that 
$$
(-1)^{\mu_F(\mathbf 0)} = (-1) ^ n + \sum_{\lambda = 0}^{n - 1} (-1) ^ {\lambda} \cdot {\# \Crit_\lambda^{out}(f)} = 1 - \sum_{\lambda = 0}^{n - 1} (-1) ^ {\lambda} \cdot {\# \Crit_\lambda^{in}(f)}
$$
\end{rem}

Sometimes one can retrieve the exact value of $\mu_F(\mathbf 0)$ from $f$, one of the instances of that is provided below.

\begin{prop}
Let $F$ be a smooth extension of $f$ to $B_{\mathbf 0}^n(1)$ such that ${\Crit F = \Morse F = \mathbf 0}$. If a point $p \in \Crit_{n - 1} f$ ($p \in \Crit_0 f$) with the largest (smallest) value of $f(p)$ lies in $\Crit_{n - 1}^{in} f$ ($\Crit_{0}^{out} f$), then $\mu_F(\mathbf 0) = n$ ($\mu_F(\mathbf 0) = 0$).
\end{prop}
\begin{proof}
The version in brackets follows from the unbracketed one by taking $-F$ instead of $F$. We prove the version without brackets.

The global maximum $M$ of $F$ is attained somewhere. If it is attained at $S^{n - 1}$, then $F(p) = M$, but $F$ grows on a curve in $B_{\mathbf 0}^n(1)$ emanating from $p$ along $v$ for any $v \in T_p^{in}B_{\mathbf 0}^n(1)$, a contradiction. Thus $M$ is attained at a critical point of $F$ in $B_{\mathbf 0}^n(1) \setminus S^{n - 1}$, but $\Crit F = \{\mathbf 0\}$, so the origin is the point of global maximum of $F$ and since it is non-degenerate, we have $\mu_F(\mathbf 0) = n$.
\end{proof}

\section{Toolbox: metamorphoses of Morse functions}
\label{Toolbox}
\begin{defn}
Let $U\subset B^n_{\mathbf 0}(1)$ be a neighbourhood of $S^{n - 1}$ and $f\colon U \to \R$ be a Morse function without critical points. We say that $f$ \emph{admits index $\lambda$} if there exists a Morse function $F\colon B^n_{\mathbf 0}(1) \to \R$ such that 
\begin{enumerate}
\item $\res{F}{U} = f$,
\item $\Crit F = \{\mathbf 0\}$, and
\item $\mu_F(\mathbf 0) = \lambda$.
\end{enumerate}
\end{defn}

Note that if $f$ admits index $\lambda$, then so does $L \circ f$ for any orientation-preserving diffeomorphism $L$ of $\R$ and $f \circ R^{-1}\colon R(U) \to \R$ for any diffeomorphism $R$ of $B^n_{\mathbf 0}(1)$. 

Note also that if $F\colon M \to \R$ is a Morse function and $A$ is a closed subset of $M$ not meeting $\Crit F$, then a vector field $V \in C^\infty(A, TM)$ satisfying $VF > 0$ can be extended to a gradient-like vector field $\tilde V\in C^{\infty}(M, TM)$ for $F$ and a Riemannian metric $g \in C^{\infty}(A, S^2T^*M)$ along $A$ can be extended to a Riemannian metric $g \in C^\infty(M, S^2T^M)$ adapted to $F$. In particular, if $F\colon B^n_{\mathbf 0}(1) \to \R$ is a Morse function, then we can assume that the Riemannian metric adapted to $F$ is the  standard metric coming from $\R^n$ outside an arbitrary neighbourhood $U$ of $\Crit F$.

\subsection{Flips}
\begin{defn}
Let $F\colon B^n_{\mathbf 0}(1) \to \R$ be a Morse function without critical points and $p \in S^{n - 1}$ be a critical point of $\res{F}{S^{n - 1}}$. A Morse function $\tilde F\colon B^n_{\mathbf 0}(1) \to \R$ is a \emph{flip of $F$ at $p$ of index $\lambda$} if
\begin{enumerate}
\item $\tilde F = F$ outside some neighbourhood $U_p$ of $p$ in $B^n_{\mathbf 0}(1)$;
\item $\grad F(p)$ and $\grad \tilde F(p)$ point in opposite directions;
\item $\res {\tilde F}{S^{n - 1}} = \res{F}{S^{n - 1}}$;
\item $\Crit \tilde F = \{\mathbf 0\}$ and $\mu_{\tilde F}(\mathbf 0) = \lambda$.
\end{enumerate}
\end{defn}

The following statement is a refinement of \cite[Lemma 1]{Bar94} by Barannikov.
\begin{lem} \label{Existence of a flip from down to up}
Let $F\colon B^n_{\mathbf 0}(1) \to \R$ be a Morse function without critical points and $p \in S^{n - 1}$ be a critical point of $f = \res{F}{S^{n - 1}}$ with $\mu_f(p) = \lambda$ and $\grad F (p) \in T_p^{in}B^n_{\mathbf 0}(1)$. Then there exists a flip $\tilde F$ of $F$ at $p$ of index $\lambda$.
\end{lem}
\begin{proof}
Without loss of generality we can assume that $F(p) = 0$. Choose a local coordinate system $(x^1, \ldots, x^{n - 1}, y)$ in a neighbourhood $U_1$ of $p$ in such a way that 
\begin{enumerate}
\item $U_1$ is given by $y \le 0$,
\item $x = (x^1, \ldots, x^{n - 1})$ is a coordinate system in $U_1 \cap S^{n - 1}$ adapted to $p$ (with respect to $f$), and
\item $F(x, y) = f(x) - 2y$ (to satisfy that first choose $\tilde y$ such that (1) and (2) hold, and then define $2y(x, \tilde y) = f(x) - F(x, \tilde y)$).
\end{enumerate}

Let $a > 0$ be such that $U_2 = B_{\mathbf 0}^{n - 1}(a) \times [-a, +\infty)$ satisfies $U_2 \cap \{y \le 0\} \subset U_1$. Note that for $(x, y) \in U_2$ with $y \le 0$ we have
$$
F(x, y) = -\sum_{i = 1}^{\lambda} \left(x^i\right) ^ 2 + \sum_{i = \lambda + 1}^{n - 1} \left(x^i\right) ^ 2 - 2y.
$$
By modifying $F$ on $\{(x, y) \in B_{\mathbf 0}^{n - 1}(a) \times (-a/5, 0) \colon a/5 < |x| < 4a/5\}$ we can obtain a smooth function $F_1$ without critical points such that $F_1(x, y) = F(x, 0)$ for ${2a/5 \le |x| \le 3a/5}$ and ${-a/10 < y \le 0}$. Now we extend this function to the smooth function 
$$
F_2 \colon  B_{\mathbf 0}^{n - 1}(a) \times (-a, 0] \cup \{(x, y) \in U_2 \colon 2a/5 \le |x| \le 3a/5\} \to \R
$$
by $F_2(x, y) = F_2(x, 0)$ for $2a/5 \le |x| \le 3a/5$ and $y > 0$.

Define a function $G\colon B_{\mathbf 0}^{n - 1}(a) \times [0, +\infty)$ by
$$
G(x, y) = 1 - \sum_{i = 1}^{\lambda} \left(x^i\right) ^ 2 + \sum_{i = \lambda + 1}^{n - 1} \left(x^i\right) ^ 2 + (y - 1) ^ 2.
$$
Note that $F_2$ and $G$ satisfy $F_2(x, 0) = G(x, 0)$ for $x \in B_{\mathbf 0}^{n - 1}(a)$ and ${\grad F_2(x, 0) = \grad G(x, 0)}$ for $|x| \le a/5$. So there is a smooth extension $F_3 \colon U_2 \to \R$ of $F_2$ such that
\begin{enumerate}
\item $F_3(x, y) = G(x, y)$ for $|x| \le a/5$ and $y \ge 1/2$ and
\item $\Crit F_3 = \Crit G = \left\{(0, \ldots, 0, 1)\right\}$.
\end{enumerate}
Now let $\tilde y\colon B_{\mathbf 0}^{n - 1}(a) \to \R$ be a smooth function satisfying
\begin{enumerate}
\item $\tilde y(x) = 2$ for $|x| \le a/5$,
\item $\tilde y(x) = 0$ for $|x| \ge a/2$, and
\item $F_3(x, \tilde y(x)) = F(x, 0)$ for all $x \in B_{\mathbf 0}^{n - 1}(a)$.
\end{enumerate}
Take a diffepmorphism $\Phi$ of $U_2$ such that
\begin{enumerate}
\item $\Phi(x, y) = (x, y)$ for $|x| \ge 2a/5$ and
\item $\Phi(x, \tilde y) = (x, 0)$ for all $x \in  B_{\mathbf 0}^{n - 1}(a)$.
\end{enumerate}
The function $F_4(x, y) = F_3(x, \Phi^-1(x, y))$ is almost the one we need. The only difference is that the critical point $q$ of $F_4$ is not at the origin. We set $\tilde F = F_4 \circ \Psi$ where $\Psi$ is a diffeomorphism of $B_{\mathbf 0}^n$ that maps $\mathbf 0$ to $q$ and is the identity near $S^{n - 1}$.
\end{proof}

\begin{cor} \label{Existence of a flip from up to down}
Let $F\colon B^n_{\mathbf 0}(1) \to \R$ be a Morse function without critical points and ${p \in S^{n - 1}}$ be a critical point of $f = \res{F}{S^{n - 1}}$ with $\mu_f(p) = \lambda - 1$ and $\grad F (p) \in T_p^{out}B^n_{\mathbf 0}(1)$. Then there exists a flip $\tilde F$ of $F$ at $p$ of index $\lambda$.
\end{cor}

\begin{proof}
It follows from the proposition that there exists a flip $\widetilde{-F}$ of $-F$ at $p$ of index $n - \lambda$. Then $-\left(\widetilde{-F}\right)$ is a flip for $F$ at $p$ of index $\lambda$.
\end{proof}

\subsection{Standard births}
\begin{defn}
Let $F\colon B^n_{\mathbf 0}(1) \to \R$ be a Morse function, $V \in C^{\infty}(S^{n - 1}, TS^{n - 1})$ be a Morse--Smale  vector field adapted to $f = \res{F}{S^{n - 1}}$, $p \in S^{n - 1}$ be a regular point of $f$, $U_p$ be a neighbourhood of $p$ in $B^n_{\mathbf 0}(1)$, $\gamma = \gamma_V^p$ be the flow line through $p$, and $\eps > 0$ be such that $\gamma(t) \in U_p$ for $|t| < 3\eps$. We say that a Morse function $\tilde F\colon B^n_{\mathbf 0}(1) \to \R$ is \emph{obtained from $F$ by a standard birth in $U_p$ of index $\lambda$} if there exists a Morse-Smale vector field $\tilde V \in C^{\infty}(S^{n - 1}, TS^{n - 1})$ for $\tilde f = \res{\tilde F}{S^{n - 1}}$ such that
\begin{enumerate}
\item $\tilde F = F$ and $\tilde V = V$ outside $U_p$;
\item $\Crit \tilde f = \Crit f \cup \{p_-, p_+\}$ where $p_{\pm} = \gamma(\pm \eps) \in U_p$;
\item $\mu_{\tilde f}(p_+) = \lambda$, $\mu_{\tilde f}(p_-) = \lambda + 1$, $\gamma^p_{\tilde V}(-\eps, \eps)$ is the unique flow line between $p_+$ and $p_-$, and $\im \gamma^{\gamma(-2\eps)}_{\tilde V} \cup \im \gamma^p_{\tilde V} \cup \im \gamma^{\gamma(2\eps)}_{\tilde V} = \im \gamma_V^p$.
\item $\grad F(x)$ and $\grad \tilde F(x)$ both lie in either $T_x^{in}{B^n_{\mathbf 0}(1)}$ or $T_x^{out}{B^n_{\mathbf 0}(1)}$ for any $x \in S^{n - 1}$.
\end{enumerate}
\end{defn}
The construction we present here is essentially the one known classically and explained by Cerf in \cite[III.1]{Cerf70}. The only difference is that we have an additional dimension, that is, we need to extend the modification of a function on $S^{n - 1}$ to its tubular neighbourhood. This is done straightforwardly, yet we write the construction in some detail as we later need its additional property, namely, that one can relate standard births at two points on the same flow line.

\begin{lem} \label{Existence of standard birth}
Let $F\colon B^n_{\mathbf 0}(1) \to \R$ be a Morse function, $V \in C^{\infty}(S^{n - 1}, TS^{n - 1})$ be a Morse--Smale vector field adapted to $f = \res{F}{S^{n - 1}}$, $p \in S^{n - 1}$ be a regular point of $f$ with $\grad F(p) \notin T_pS^{n - 1}$, and $\lambda \in \{0, \ldots, n - 2\}$ be a number. Then for any sufficiently small neighbourhood $U_p$ of $p$ and sufficiently small $\eps > 0$ there exists a Morse function $\tilde F\colon B^n_{\mathbf 0}(1) \to \R$ obtained from $F$ by a standard birth in $U_p$ of index $\lambda$.
\end{lem}
\begin{proof}
Without loss of generality we can assume that $F(p) = 0$. Choose a local coordinate system $(x^1, \ldots, x^{n - 1}, y)$ in a neighbourhood $U_1$ of $p$ such that
\begin{enumerate}
\item $U_1$ is given by $y \le 0$;
\item $F(x^1, \ldots, x^{n - 1}, y) = x^{n - 1} + \sigma y$ with $\sigma = 1$ if $\grad F(p) \in T_p^{out}B^n_{\mathbf 0}(1)$ and $\sigma = -1$ if ${\grad F(p) \in T_p^{in}B^n_{\mathbf 0}(1)}$;
\item the flow line through $p$ is given by $\gamma(t) = (0, \ldots, 0, t, 0)$.
\end{enumerate}

Let $a > 0$ be such that $U_2 = B_{\mathbf 0}^{n - 1}(a) \times (-a, a)$ satisfies $U_2 \cap \{y \le 0\} \subset U_1$. Let $omega\colon \R \to \R$ be a smooth function that is equal to $1$ for $|x| < a/4$, equal to $0$ for $|x| > a/2$, symmetric with respect to $0$ and monotone on $[0, +\infty)$. Define a function
$$
G(x, y) = -\sum_{i = 1}^\lambda \left(x^i\right) ^ 2 + \sum_{i = \lambda + 1}^{n - 2} \left(x^i\right)^2  + \left(x ^ {n - 1}\right)^3 + (1 - 2\omega(|(x, y)|))\eps_1 x^{n - 1} + \sigma y
$$
and a diffeomorphism 
$$
\Phi(x, y) = (x^1, \ldots, x^{n - 2}, -\sum_{i = 1}^\lambda \left(x^i\right) ^ 2 + \sum_{i = \lambda + 1}^{n - 2} \left(x^i\right)^2  + \left(x ^ {n - 1}\right)^3 + \eps_1 x^{n - 1}, y),
$$
where $\eps_1 = 3 \left(\frac{3 \eps}{4}\right) ^ {2/3}$.

Then $\tilde F = G \circ \Phi^{-1}$ and $\tilde f = \res{\tilde F}{\{y = 0\}}$ have the following properties
\begin{enumerate}
\item $\tilde F(x, y) = x^{n - 1} + \sigma y = F(x, y)$ if $|x| > a/2$ or $|y| > a/2$;
\item $\Crit \tilde F = \varnothing$;
\item $\Crit \tilde f = \{p_+, p_-\} = \{(0, \ldots, 0, \eps, 0), (0, \ldots, 0, -\eps, 0)\}$; $\mu_{\tilde f} (p_+) = \lambda$, and ${\mu_{\tilde f}(p_-) = \lambda + 1}$;
\item $\frac{\partial \tilde F}{\partial y}(x, y) = \sigma = \frac{\partial F}{\partial y}(x, y)$ for all $(x, y) \in U_2$.
\end{enumerate}
Thus $\tilde F$ and $\tilde V = \grad \tilde f$ are the desired function and vector field. 
\end{proof}

\begin{rem}
From the construction it follows that $\tilde f$ and $\tilde V$ depend only on $f$ and $V$ and not on the extension of $f$ to $F$.
\end{rem}

\subsection{Connecting two Morse functions}
First we prove a technical lemma that allows us to obtain a Morse function without critical points in a neighbourhood of $S^{n - 1}$ inside $\R^n$ from two Morse functions without critical points in neighbourhoods of $S^{n - 1}$ inside $\{x \in \R^{n} \colon |x| \le 1\}$ and $\{x \in \R^n \colon |x| \ge 1\}$.

\begin{lem} \label{Smoothing in a neighbourhood of the sphere}
Let $F$ be a continuous function on $S_\eps = \{x \in \R^n \colon 1-\eps < |x| < 1 + \eps\}$ which is a Morse function without critical points on $S_\eps^+ = \{x \in \R^n \colon 1 \le |x| < 1 + \eps\}$ and $S_\eps^- = \{x \in \R^n \colon 1 - \eps < |x| \le 1\}$. Suppose that $\grad\left( \res{F}{S_\eps^+}\right)(p) = \grad\left( \res{F}{S_\eps^-}\right)(p)$ for all $p \in \Crit \res{F}{S^{n - 1}}$. Then there exists a Morse function without critical points $\tilde{F}$ on $S_\eps$ such that $\tilde F$ is equal to $F$ in a neighbourhood of $\partial S_\eps$.
\end{lem}
\begin{proof}
Let $F^\pm$ be smooth extensions of $F$ from $S_\eps^\pm$ to $S_\eps$ such that ${\grad\left( \res{F}{S_\eps^\pm}\right)(x) = \grad\left(F^\pm\right)(x)}$ for all $x \in S^{n - 1}$. Taking a smaller $\eps$ we can assume that $F^{\pm}$ are Morse functions without critical points. Let ${c = \min_{y \in S^{n - 1}} \min\left\{\left|\grad \res{F}{S_\eps^+} (y)\right|, \left|\grad \res{F}{S_\eps^-} (y)\right|\right\}}$, for $p \in \Crit \res{F}{S^{n - 1}}$ let $U_p$ be a small cap on $S^{n - 1}$ around $p$ such that $\left|\grad F^{+}(x) - \grad F^{-}(x)\right| < c/100$ and $\left|\grad F^{\pm}(x) - \grad F^{\pm}(p)\right| < c/100$ for any $x \in U_p$. Pick $\delta > 0$ such that the following conditions are satisfied
\begin{enumerate}
\item $F_{\alpha^+, \alpha^-}^t = \alpha^+ \res{F^+}{tS^{n - 1}} + \alpha^- \res{F^-}{tS^{n - 1}}$ is a Morse function equivalent to $\res{F}{S^{n - 1}}$ for any ${t \in [1 - \delta, 1 + \delta]}$ and $\alpha^+, \alpha^- > 0$ with $\alpha^+ + \alpha^- = 1$ and for any $q \in \Crit F_{\alpha^+, \alpha^-}^t$ the point $q/|q|$ lies in $U_p$ for the corresponding $p \in \Crit \res{F}{S^{n - 1}}$;
\item $\left|\grad F^{\pm} (x) - \grad F^{\pm} (x/|x|) \right| < c/100$ for any any $x \in S_{\delta}$.
\end{enumerate}
Take a partition of unity $\{\ph^+, \ph^-\}$ on $S_\eps$ subordinate to the open cover ${\big\{\{x \in \R^n \colon 1 - \delta \le |x| \le 1 + \eps\}, \{x \in \R^n \colon 1 - \eps \le |x| \le 1 + \delta\}\big\}}$ such that $\ph^+$ constant on each $|x|S^{n-1}$, decreases in $|x|$ and $\left|\grad \ph^+\right| < 2/\delta$. We define $\tilde F = \ph^+F^+ + \ph^-F^-$. We need to prove that $\tilde F$ has no critical points. 

Consider $x \in S_\eps$. If $x \notin S_\delta$, then $\tilde F$ coincides with $F^+$ or $F^-$ in a neighbourhood of $x$, therefore, $x$ is not a critical point of $\tilde F$. Now let $x \in S_\delta$. If $x$ is not a critical point of $\res{\tilde F}{|x|S^{n - 1}}$, then $x$ is obviously not a critical point of $\tilde F$. If $x \in \Crit \res{\tilde F}{|x|S^{n - 1}}$, then let $p$ be the corresponding critical point of $\res{F}{S^{n - 1}}$. We have
\begin{align*}
\left|\grad \tilde F(x)\right| &= \left|\grad F^-(x) + (F^+(x) - F^-(x)) \cdot \grad \ph^+ (x)  + \ph^+(x) \cdot (\grad F^+ (x) - \grad F^- (x))\right|\\
&\ge \left| \grad F^-(p) \right| - \left| \grad F^-(p) - \grad F^-(x/|x|) \right| - \left| \grad F^-(x) - \grad F^-(x/|x|) \right|\\
& - \left| \grad \ph^+ (x) \right| \cdot \left| x - x/|x|\right| \cdot \max_{y \in [x, x/|x|]} \left| \grad F^+(y) - \grad F^-(y)\right|\\
&- \left| \grad F^+ (x) - \grad F^- (x)\right|
\end{align*}
For $y \in [x, x/|x|]$ we have 
\begin{align*}
\left| \grad F^+(y) - \grad F^-(y)\right| &\le \left| \grad F^+(y) - \grad F^+(y/|y|)\right|\\
&+ \left| \grad F^+(x/|x|) - \grad F^-(x/|x|)\right| + \left| \grad F^-(y/|y|) - \grad F^-(y)\right|\\
&\le 3c/100
\end{align*}
so
$$
\left|\grad \tilde F(x)\right| \ge c - c/100 - c/100 - \frac{2}{\delta} \cdot \delta \cdot \frac{3c}{100} - 3c/100 > 0
$$
thus $x$ is not a critical point of $\tilde F$.
\end{proof}

Now we connect two Morse functions in such a way that the result will satisfy the assumptions of the previous lemma.
\begin{lem} \label{Connecting two Morse functions}
Let $f_1, f_2\colon S_{\eps}^{-} \to \R$ be two Morse functions without critical points. Suppose that $\Crit \res{f_1}{S^{n - 1}} = \Crit \res{f_2}{S^{n - 1}} = C$, $f_1(p) \ne f_2(p)$ for each $p \in C$, and the vectors $\grad f_i(p)$ both point out if $f_2(p) > f_1(p)$ and point in if $f_2(p) < f_1(p)$. If there exists a smooth vector field $V \in C^{\infty}(S^{n - 1}, TS^{n - 1})$ such that $Vf_i(x) > 0$ for any $x \in S^{n - 1} \setminus C$, then there exists a Morse function without critical points $F\colon \{x \in \R^n \colon 1 \le |x| \le 2\} = A \to \R$ such that
\begin{enumerate}
\item $F(i\cdot x) = f_i(x)$ for any $x \in S^{n - 1}$;
\item $i \cdot \grad F (i \cdot p) = \grad f_i(p)$ for any $p \in C$;
\end{enumerate} 
\end{lem}
\begin{proof}
For each $x \in S^{n - 1}$ denote by $n_x$ the unit outer normal to $S^{n - 1}$ at $x$. For $p \in C$ let $U_p$ be a small cap in $S^{n - 1}$ around $p$ such that 
$$
\sign(f_2(x) - f_1(x)) = \sign \left\langle n_x, \grad f_1(x)\right\rangle = \sign \left\langle n_x, \grad f_2(x)\right\rangle \text{for all } x \in U_p.
$$
Let $\ph_p \colon \{x \in A \colon x/|x| \in U_p\} \to \R$ be a smooth function such that for any $x \in S^{n - 1}$
\begin{enumerate}
\item $\ph_p(x) = f_1(x)$ and $\grad \ph_p(x) = \grad f_1(x)$;
\item $\ph_p(2x) = f_2(x)$ and $2\grad \ph_p(2x) = \grad f_2(x)$;
\item $\res{\ph_p}{[x, 2x]}$ is strictly monotone.
\end{enumerate} 
Let $F_{reg}\colon A \to \R$ be a convex combination of $f_1$ and $f_2$:
$$
F_{reg}(x) = (2 - |x|)f_1(x/|x|) + (|x| - 1) f_2(x/|x|).
$$
Take a partition of unity $\{h_p\}_{p \in C} \cup \{h\}$ subordinate to an open cover $\{U_p\}_{p \in C} \cup \{U\}$ where $U = S^{n - 1} \setminus C$ and define
$$
F = hF_{reg} + \sum_{p \in C} h_p\ph_p.
$$
We need to check that $F$ has no critical points. Indeed, $\langle \grad F(x), x/|x|\rangle = 0$ only if $f_1(x) = f_2(x)$ and that can happen only outside each of $U_p$. But then for $y \in S^{n - 1}$ we have
$$
F(|x| \cdot y)= (2 - |x|)f_1(y) + (|x| - 1)f_2(y),
$$ 
so 
$$
VF(|x| \cdot \underline{\phantom{y}}) (y) = ((2 - |x|)) Vf_1(y) + (|x| - 1)Vf_2(y) > 0
$$
thus $\grad F(x) \ne 0$.
\end{proof}

\section{Functions in a neighbourhood of $S^{n - 1}$ admitting different indices}
\label{Functions admitting different indices}
With all the tools developed in the previous section we are ready to construct a function that admits two different indices. We first illustrate the idea with the two-dimensional case.

\begin{figure}
\includegraphics[scale=0.4]{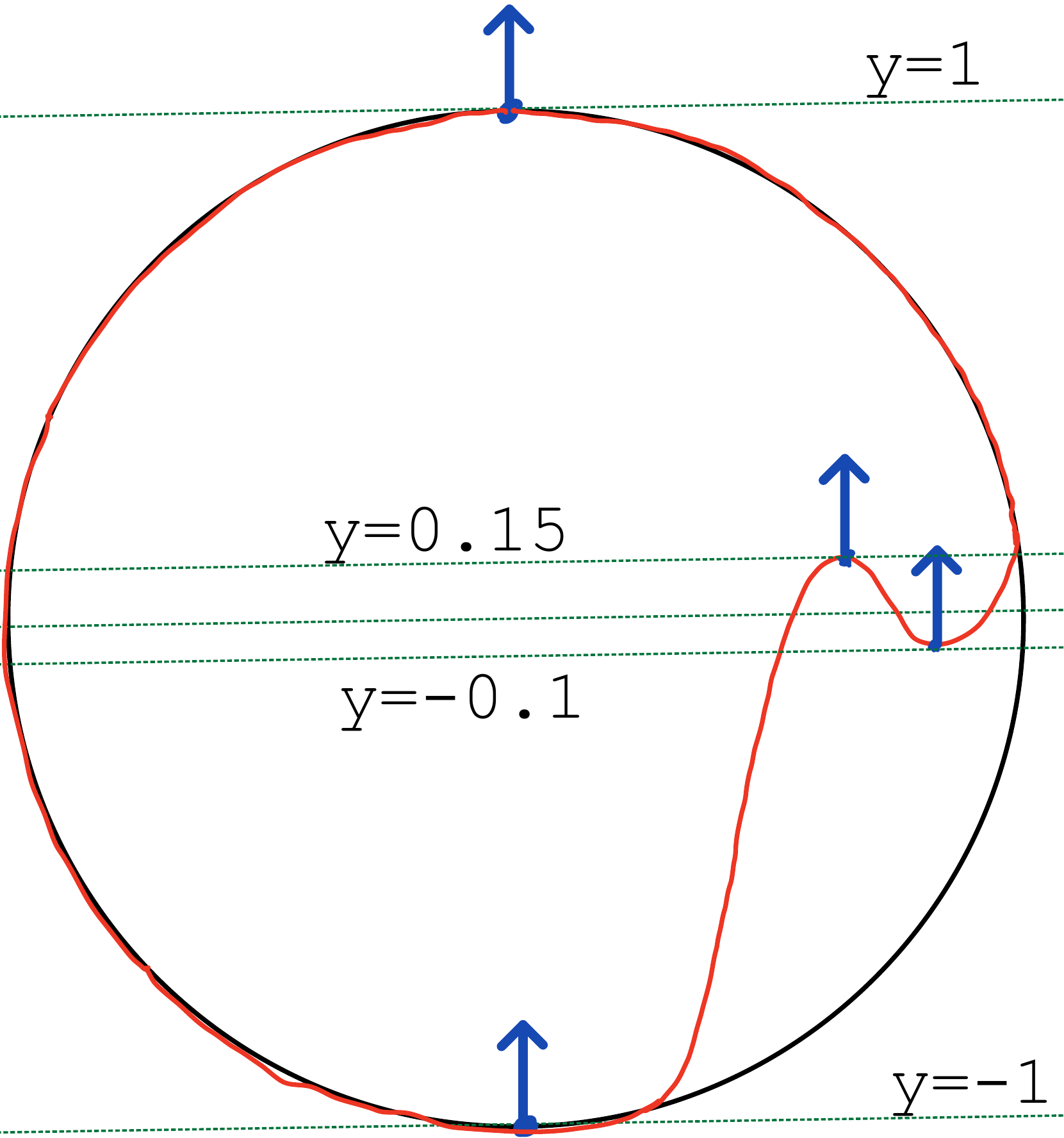}
\caption{Insertion of a pair of critical points with arrows down in dimension $2$}
\label{F^0_2}
\end{figure}

\begin{figure}
\includegraphics[scale=0.4]{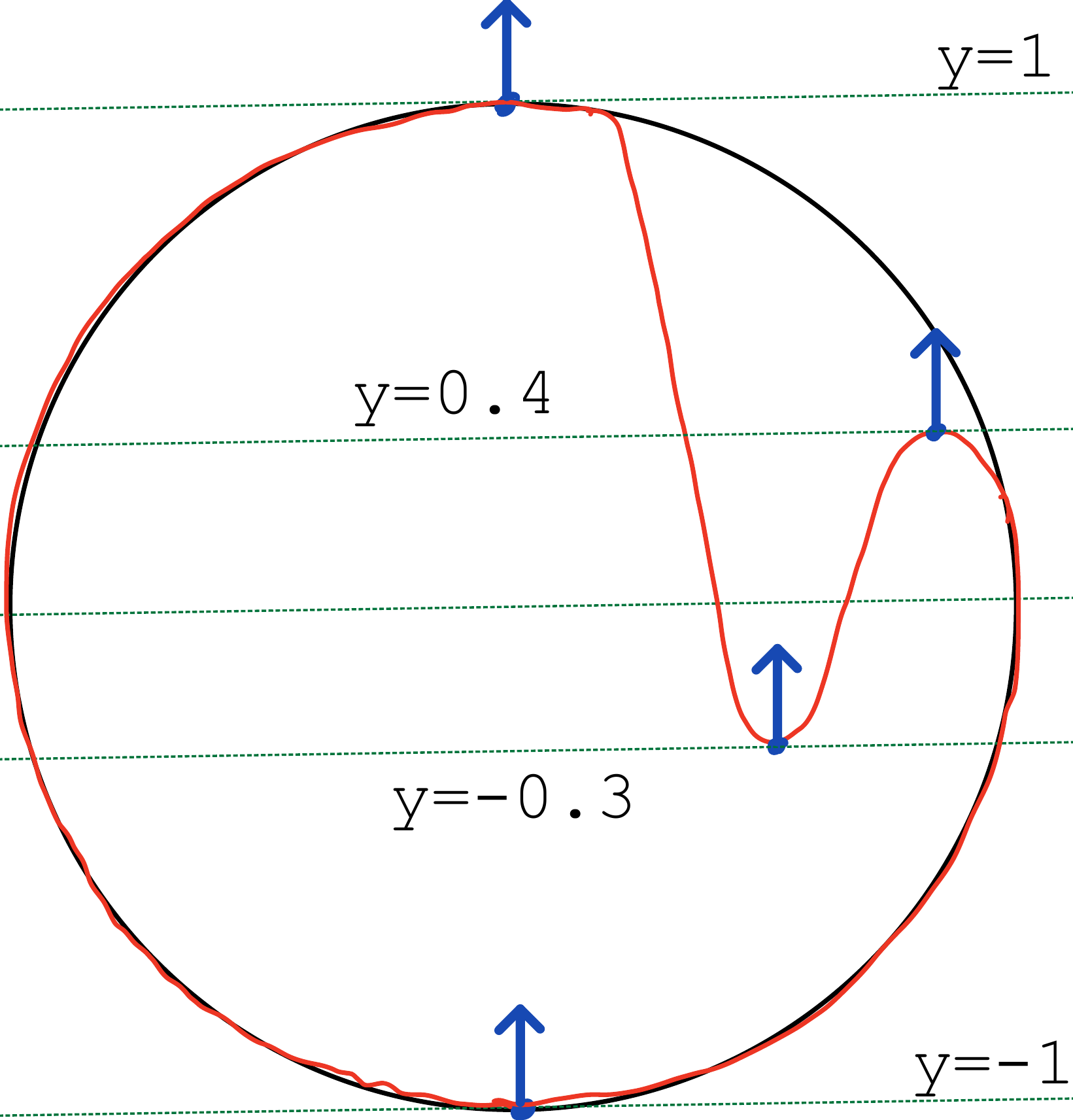}
\caption{Insertion of a pair of critical points with arrows down in dimension $2$}
\label{F^2_2}
\end{figure}

\begin{figure}
\includegraphics[scale=0.35]{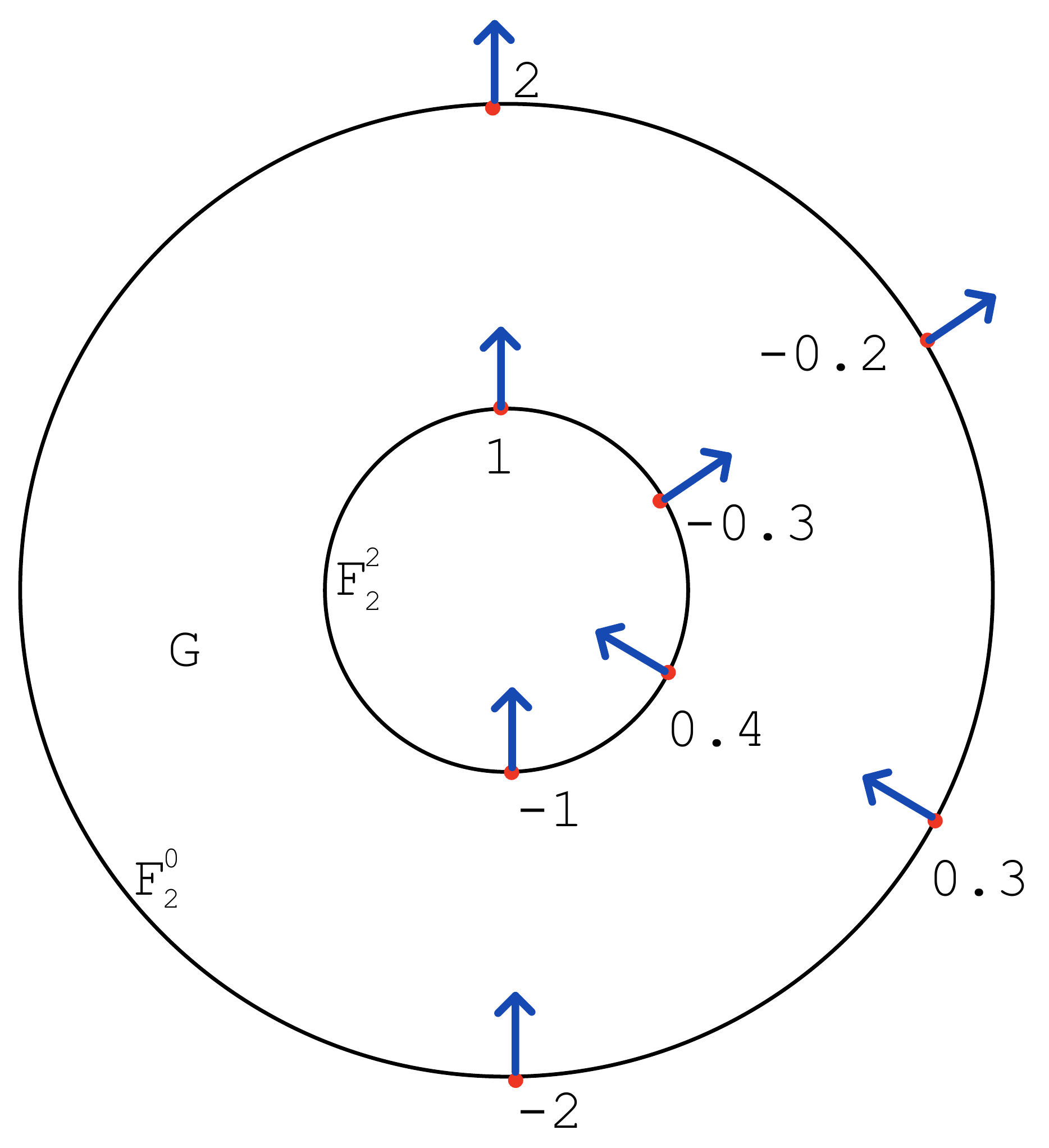}
\caption{Connecting two functions in dimension $2$}
\label{connection_2}
\end{figure}

Let $F_1^0 \in C^{\infty}(B_{\mathbf 0}^2(1))$ be a Morse function without critical points given by ${F_1^0(x, y) = 2y(\Phi^0(x, y))}$, where $\Phi^0$ is a diffeomorphism of the plane that maps $B_{\mathbf 0}^2$ diffeomorphically onto the figure bounded by the red curve on Figure~\ref{F^0_2}. Let $F_1^2 \in C^{\infty}(B_{\mathbf 0}^2(1))$ be a Morse function without critical points given by $F_1^0(x, y) = y(\Phi^2(x, y))$, where $\Phi^2$ is a diffeomorphism of the plane that maps $B_{\mathbf 0}^2$ diffeomorphically onto the figure bounded by the red curve on Figure~\ref{F^2_2}. Which choose the diffeomorphisms $\Phi^0$ and $\Phi^2$ in such a way that $\Crit \res{F^0_1}{S^1} = \Crit \res{F^2_1}{S^1}$.

$F_2^0$ is obtained from $F_1^0$ by a flip of index $0$ at the critical point $p^0$ with $F_1^0(p^0) = -0.2$. $F_2^2$ is obtained from $F_1^2$ by a flip of index $2$ at the critical point $p^2$ with $F_1^2(p^2) = 0.4$. Now, using Lemma~\ref{Connecting two Morse functions}, we connect $F_2^0$ and $F_2^2$, let $G$ be the resulting function in $\{p \in \R^2 \colon 1 \le |p| \le 2\}$ (see Figure~\ref{connection_2}). 

$G$ and $F_2^2$ satisfy the assumptions of Lemma~\ref{Smoothing in a neighbourhood of the sphere}, so there exists a Morse function function $F_3^2 \in C^{\infty}(B_{\mathbf 0}^2(2))$ such that $\Crit F_3^2 = \{\mathbf 0\}$, $F_3^2$ is equal to $F_2^2$ in a neighbourhood of the origin and is equal to $G$ in a neighbourhood of $2S^{1}$. 

Now extend $F_2^0$ from $B_{\mathbf 0}^2(2)$ to a function $F_3^0$ on $B_{\mathbf 0}^2(2 + \eps)$ in such a way that there exist a diffeomorphism $\Psi\colon B^2_{\mathbf 0}(1) \to B^2_{\mathbf 0}(1 + \eps)$ isotopic to the homothety and an orientation-preserving diffeomorphism $L\colon \R \to \R$ such that $F_3^0 \circ \Psi = L \circ F_2^0$ near $S^{1}$. $F_3^2$ and $F_3^0$ satisfy the assumptions of Lemma~\ref{Smoothing in a neighbourhood of the sphere} near $2S^1$, so there exists a Morse function $F$ on $B_{\mathbf 0}^2(2 + \eps)$ such that $\Crit F \{\mathbf 0\}$, $F$ is equal to $F_3^2$ in a neighbourhood of the origin and to $F_3^0$ in a neighbourhood $U$ of $(2 + \eps) S^1$. 

We obtained a function $\res{F}{U}$ that can be extended to $B_{\mathbf 0}^2(2 + \eps)$ in two different ways: by $F$ and by $F_3^0$. In the former case we have $\Crit F = \Morse F = \{\mathbf 0\}$ and $\mu_F(\mathbf 0) = 2$, and in the latter one we have $\Crit F_3^0 = \Morse F_3^0 = \{\mathbf 0\}$ and $\mu_F(\mathbf 0) = 0$. Now let us proceed with the general case.

\begin{thm} \label{two functions}
Let $n \ge 2$ and $0 \le \lambda \le n - 2$. Then there exists a Morse function $f\colon U \to \R$ in the neighbourhood $U$ of $S^{n - 1}$ inside $B_{\mathbf 0}^n(1)$ that admits both indices $\lambda$ and $\lambda + 2$.
\end{thm}
\begin{proof}
The proof essentially follows the strategy explained above for the two-dimensional case.
\begin{enumerate}
\item $F^\lambda_0 = F^{\lambda + 2}_0 = x^n$. The gradient-like vector fields $V = \grad \res{x^n}{S^{n - 1}}$ for the standard round metric on the sphere, $\Phi_t$ is the flow on $S^{n - 1}$ generated by $-V$. The critical points of $f^{\lambda}_0$ and $f^{\lambda + 2}_0$ are $n = (0, \ldots, 0, 1)$ and $s = (0, \ldots, 0, -1)$.
\item Let $\gamma$ be the flow line on $S^{n - 1}$ through the point $p = (1, 0, \ldots, 0)$. Denote $p_+ = \gamma(\eps)$, $p_- = \gamma(-\eps)$.
\item Set $F_1^{\lambda} = F_0^{\lambda} \circ \Psi^\lambda$ and $F_1^{\lambda + 2} = F_0^{\lambda + 2} \circ \Psi^{\lambda + 2}$, where $\Psi^\lambda$ and $\Psi^{\lambda + 2}$ are diffeomorphisms of $B_{\mathbf 0}^n(1)$ such that $\res{\Psi^\lambda}{S^{n - 1}} = \Phi_{2\eps}$ and $\res{\Psi^{\lambda + 1}}{S^{n - 1}} = \Phi_{-2\eps}$.
\item $F_2^{\lambda}$ and $F_2^{\lambda + 2}$ are obtained by a standard birth at point $p$ from $F_1^{\lambda}$ and $F_1^{\lambda + 2}$ respectively. 
\item $F_3^\lambda$ is a flip of $F_2^\lambda$ at $p_-$ of index $\lambda$, $F_3^{\lambda}$ is a flip of $F_3^{\lambda + 2}$ at $p_+$ of index $\lambda + 2$. 
\item Let $L\colon \R \to \R$ be an orientation preserving deffeomorphism such that 
$$
L(f_3^\lambda(s)) < f_3^{\lambda + 2}(s) < f_3^{\lambda + 2}(p_-) < L(f_3^\lambda(p_-)) < L(f_3^\lambda(p_+)) < f_3^{\lambda + 2}(p_+) < f_3^{\lambda + 2}(n) < L(f_3^\lambda(n)).
$$
Set $F_4^\lambda = L \circ F_3^\lambda$ and $F_4^{\lambda + 2} = F_3^{\lambda + 2}$.
\item Use Lemma~\ref{Connecting two Morse functions} to connect $\res{F_4^\lambda}{S_\eps^-}$ and $\res{F_4^{\lambda + 2}}{S_\eps^-}$ by a function $G\colon \{1 \le |x| \le 2\} \to \R$.
\item Extend $F_4^\lambda$ to $F_5^\lambda$ on $B^n_{\mathbf 0}(1 + \eps)$ in such a way that there exist a diffeomorphism $\Psi\colon B^n_{\mathbf 0}(1) \to B^n_{\mathbf 0}(1 + \eps)$ isotopic to the homothety and an orientation-preserving diffeomorphism $L\colon \R \to \R$ such that $F_5^\lambda \circ \Psi = L \circ F_4^\lambda$ near $S^{n - 1}$.
\item The functions $\res{F_4^{\lambda + 2}}{S_\eps^-}$ and $\res{G}{S_\eps^+}$ satisfy the assumptions of Lemma~\ref{Smoothing in a neighbourhood of the sphere}, let $F_5^{\lambda+2}$ be the resulting Morse function in $B^n_{\mathbf 0}(2)$.
\item The function $\res{F_5^{\lambda + 2}}{2S_\eps^-}$ and $\res{F_5^{\lambda}}{2S_\eps^+}$ satisfy the assumptions of Lemma~\ref{Smoothing in a neighbourhood of the sphere}, let $F_6^{\lambda + 2}$ be the resulting Morse function in $B^n_{\mathbf 0}(2 + 2\eps)$.
\end{enumerate}
The functions $F_5^{\lambda}$ and $F_6^{\lambda + 2}$ are equal in a neighbourhood of $\partial B^n_{\mathbf 0}(2 + 2\eps)$, ${\Crit F_5^{\lambda} = \Crit F_6^{\lambda + 2} = \{\mathbf 0\}}$, $\mu_{F_5^{\lambda}}(\mathbf 0) = \lambda$, and $\mu_{F_6^{\lambda + 2}}(\mathbf 0) = \lambda + 2$. Thus we constructed the desired functions.
\end{proof}

\bibliographystyle{acm}
\bibliography{Morse_biblio}

\begin{thebibliography}{1}

\bibitem{Bar94}
{\sc Barannikov, S.}
\newblock The framed morse complex and its invariants.
\newblock {\em American Mathematical Society Translations 2\/} (1994).

\bibitem{Cerf70}
{\sc Cerf, J.}
\newblock La stratification naturelle des espaces de fonctions
  diff\'erentiables r\'eelles et le th\'eor\`eme de la pseudo-isotopie.
\newblock {\em Publications Math\'ematiques de l'IH\'ES 39\/} (1970), 5--173.

\bibitem{Milnor65}
{\sc Milnor, J., Siebenmann, L., and Sondow, J.}
\newblock {\em Lectures on the H-Cobordism Theorem}.
\newblock Princeton University Press, 1965.

\bibitem{Milnor69}
{\sc Milnor, J., SPIVAK, M., and WELLS, R.}
\newblock {\em Morse Theory}, vol.~51 of {\em Annals of Mathematics Studies}.
\newblock Princeton University Press, 1969.

\bibitem{Nic11}
{\sc Nicolaescu, L.}
\newblock {\em An Invitation to Morse Theory}.
\newblock Springer New York, 2011.

\bibitem{Sharko93}
{\sc Sharko, V.}
\newblock {\em Functions on Manifolds: Algebraic and Topological Aspects:
  Algebraic and Topological Aspects}.
\newblock Translations of mathematical monographs. American Mathematical
  Society, 1993.

\bibitem{Smale61}
{\sc Smale, S.}
\newblock On gradient dynamical systems.
\newblock {\em Annals of Mathematics 74}, 1 (1961), 199--206.

\end{thebibliography}

\end{document}